\documentclass[12pt]{amsart}
\usepackage[letterpaper, margin=1.25in]{geometry}
\usepackage{amsmath,amsthm,amscd,color}
\usepackage{amssymb}
\usepackage{amsfonts}
\usepackage{latexsym}
\usepackage{mathrsfs}
\usepackage{hyperref}
\usepackage{geometry}
\usepackage{fancyvrb}
\usepackage{dsfont}
\usepackage[dvipsnames]{xcolor}
\usepackage{tikz-cd}
\usepackage{rotating}
\usepackage{graphicx,diagbox}
\usepackage{caption}
\usepackage{subcaption}
\usepackage{MnSymbol}
\usepackage{stmaryrd}

\newtheorem{theorem}[equation]{Theorem}
\newtheorem{lemma}[equation]{Lemma}
\newtheorem{proposition}[equation]{Proposition}

\theoremstyle{remark}
\newtheorem{definition}[equation]{Definition}

\numberwithin{equation}{section}
\numberwithin{table}{section}

\renewcommand{\frak}{\mathfrak}

\DeclareMathOperator{\CR}{CR}

\newcommand{\Z}{{\mathbb{Z}}}
\newcommand{\R}{{\mathbb{R}}}
\newcommand{\C}{{\mathbb{C}}}

\newcommand{\F}{{\mathbb{F}}}

\newcommand{\TT}{{\mathcal{T}}}
\newcommand{\FF}{{\mathcal{F}}}

\newcommand{\CC}{{\mathcal{C}}}
\newcommand{\LL}{{\mathcal{L}}}

\title{Khovanov Homology in Connected Sums, Properties and Applications}
\author{Alan Du}

\begin{document}

\maketitle

\begin{abstract}
We extend the definition of Khovanov-Lee homology to links in connected sums of interval bundles over surfaces and $S^1\times S^2$'s, and construct a Rasmussen-type invariant for links in these manifolds. As an application, we prove an inequality relating the Rasmussen-type invariant to the genus of surfaces with boundary in four-manifolds that are boundary connect sums of $D^2\times S^2$, $\C P^2\setminus B^4$, and $DTS^2$.
\end{abstract}

\section{Introduction}

Khovanov homology is a homology theory for links in $S^3$ that categorifies the Jones polynomial \cite{khovanov-categorification}. Lee \cite{lee-hlee} constructed a deformation of Khovanov homology that was used by Rasmussen \cite{rasmussen-s-inv} to define a numerical invariant for knots, called the $s$-invariant, that gives a bound for the slice genus of a knot.

An ongoing problem is to generalize Khovanov homology to links in other $3$-manifolds. So far, this has been done for orientable interval-bundles over surfaces by \cite{aps-i-bundles}, $\R P^3$ by \cite{gabrovsek-rp3}, and $\#^q(S^1\times S^2)$ by \cite{willis-kh}. 

This paper builds off of previous work of the author, who constructed Khovanov homology for links in $3$-manifolds $M=M_1\#\dots\# M_r$ where $M_i$ is an orientable interval bundle over some surface $F_i$ for $i=1,\dots, r$. First, we show how to generalize this construction to manifolds $M^q=M\#(\#^q (S^1\times S^2))$ to get a Khovanov chain complex $CKh(L)$ for links in $M^q$.

\begin{theorem}\label{thm:kh-inv-s1s2}
For $\LL$ a link in $M^q$, and $L$ a diagram of $\LL$, the bigraded chain homotopy type of $CKh(L)$ is an invariant of the link $\LL$ under isotopy.
\end{theorem}

Similar to \cite{willis-kh}, $CKh(L)$ is constructed by inserting complexes $C^*(\TT_{n_i}^\infty)$ into the diagram, where $C^*(\TT_{n_i}^\infty)$ can be seen as the Bar-Natan complex for the formal infinite twist on $n_i$ strands. This complex is realized as the stable colimit of a sequence of full twist complexes $C^*(\TT_{n_i}^k)$ as $k$ goes to infinity, and just as in \cite{willis-kh}, $CKh(L)$ can be approximated by these finite complexes.

Next, we define the Lee deformation $CLee(L)$ of the Khovanov chain complex and construct the Lee generators, which generate Lee homology $HLee(L)$. Unlike the Lee generators for links in $S^3$, our generators are somewhat non-canonical and depend on additional choices. However, the Rasmussen $s$-invariant, which we define using the Lee generators, does not depend on these choices.

\begin{theorem}\label{thm:s-inv}
The $s$-invariant of an oriented link $s(L)$ is an invariant of the oriented link $\LL\subset M^q$ under isotopy.
\end{theorem}

In addition to $S^3$, the $s$-invariant was also defined for links in $\R P^3$ by \cite{mw-s-rp3} and in $\#^q(S^1\times S^2)$ by \cite{mmsw-s-invariant}. In the case that $\LL\subset M^q$ is a local $S^3$, $\R P^3$, or $\#^q(S^1\times S^2)$ link, that is, $\LL$ is entirely contained in a $S^3$, $\R P^3$, or $\#^q(S^1\times S^2)$ connect summand of $M^q$, then our $s$-invariant $s(\LL)$ agrees with the previously defined $s$-invariant of $\LL$ considered as a link in $S^3$, $\R P^3$, or $\#^q(S^1\times S^2)$, respectively.

For an oriented cobordism $\Sigma$ in $S^3\times I$ between oriented links $L_0$ and $L_1$, there exists a cobordism map $\phi_\Sigma:HLee(L_0)\to HLee(L_1)$ that is quantum filtered of degree $\chi(\Sigma)$, where $\chi$ is the euler characteristic. If the cobordism is weakly-connected, meaning that every component of $\Sigma$ has a boundary component in $L_0$, then the cobordism map preserves Lee generators, and as a result, the $s$-invariants of $L_0$ and $L_1$ give a bound on the euler characteristic of $\Sigma$.

The same is true for cobordisms $\Sigma$ in $M^q\times I$:

\begin{theorem}\label{thm:cobordism-s-ineq}
Let $\Sigma\subset M^q\times I$ be a cobordism from link $L_0$ to link $L_1$. If every component of $\Sigma$ has a boundary component in $L_0$, then
$$s(L_1)-s(L_0)\ge \chi(\Sigma).$$
\end{theorem}

Lastly, we give an application of the $s$-invariant towards slice genus bounds. Let $D(d)$ denote the $D^2$-bundle over $S^2$ with euler number $d$. For example, $D(0)=D^2\times S^2$ with boundary $S^1\times S^2$, $D(1)=\C P^2\setminus B^4$ with boundary $S^3$, and $D(2)=DTS^2$ with boundary $\R P^3$.

\begin{theorem}\label{thm:slice-genus-bound}
Let $X=X_1\natural\dots\natural X_n$, where $X_i=D(d_i)$ for $i=1,\dots,n$, $d_i\in\{0,1,2\}$. Let $\Sigma\subset X$ be a null-homologous properly embedded orientable connected surface with boundary a knot $K\subset\partial X$, then
$$2g(\Sigma)\ge -s(K).$$
\end{theorem}

The above theorem generalizes \cite[Theorem 1.15, Corollary 1.9]{mmsw-s-invariant} and the theorem of \cite{ren-slice-genus}.

In Section 2, we establish notation and outline the construction of Khovanov homology in connected sums. In Section 3, we define Khovanov homology in connected sums with $S^1\times S^2$ and prove Theorem~\ref{thm:kh-inv-s1s2}. In Section 4, we define Lee homology and the $s$-invariant and prove Theorem~\ref{thm:s-inv} and Theorem~\ref{thm:cobordism-s-ineq} as well as some additional properties of the $s$-invariant. In Section 5, we prove Theorem~\ref{thm:slice-genus-bound}.

\textbf{Acknowledgments} The author would like to thank Yi Ni for advising him in this project and Daren Chen, Qiuyu Ren, Hongjian Yang, and Gheehyun Nahm for helpful conversations.

\textbf{AI Declaration} No AI was used in the preparation of any part of this article.

\section{Khovanov Homology in Connected Sums of Interval Bundles}

We will use the following notation.

\begin{itemize}
    \item Let $R$ be an arbitrary commutative, unital ring. When we discuss Lee homology, we will require that $2$ is invertible. 
    \item Let $M$ denote the oriented $3$-manifold $M=M_1\#\dots\#M_r$, where $M_i$ is an oriented $I$-bundle over the surface $F_i$.
    \item We fix $r-1$ separating spheres to cut $M$ into its connected sum components. On the diagram level, this is the same as choosing disjoint disks $D_i^{\pm}$ for $i=1,\dots,r-1$, where $D_i^-\subset F_i$ and $D_i^+\subset F_{i+1}$.
    \item A tangle diagram $T$ for a tangle $\TT\subset M$ is a projection of the tangle drawn in
    $$\Sigma=\left(\bigsqcup_{i=1}^r F_r^\circ\right)/(\partial D_i^-\sim\partial D_i^+)\subset M$$ as in Figure~\ref{fig:link-example}. We isotope the tangle to be disjoint from $D_i^\pm\times I$ except for the intersections of the tangle with the separating spheres that occur at $\partial D_i\times\{1/2\}$. Let $\CR(T)$ denote the set of crossings of $T$.
    \item A resolution $r$ for a tangle diagram $T$ is an assignment $r:\CR(T)\to\{0,1\}$. The resolution diagram $T_r$ is the crossingless tangle diagram obtained by replacing each crossing of $T$ with either the $0$ or $1$-smoothing of the crossing according to $r$.
    \item Let $\CC(M,B)=Kom(Mat(Cob^3_{/l}(M,B)))$ denote the Bar-Natan category consisting of complexes of formal linear combinations of crossingless tangle diagrams in $M$ with boundary $B$, where morphisms are matrices of formal linear combinations of cobordisms between the tangles, modulo local relations in \cite[Section 4.1.2]{bar-natan}.
    \item Let $BN(T)$ denote the Bar-Natan complex of $T$, the object of $\CC(M,\partial T)$ as constructed in Section 3 of \cite{du-kh-connect-sums}.
    \item For a link diagram $L$ of a link in $M$, let $CKh(L)$ denote the Khovanov chain complex of $L$, which is obtained by applying the Khovanov functor $\FF$ to $BN(L)$. That is, let $V$ be the two-dimensional graded $R$-module $V=\langle v_+,v_-\rangle$, where the quantum degree of $v_\pm$ is $\pm 1$. For a crossingless link diagram consisting of $n$ disjoint circles, the Khovanov functor assigns $V^{\otimes n}$, and for a cobordism, the Khovanov functor assigns an $R$-linear map as described in Section 4 of \cite{du-kh-connect-sums}.
    \item As in \cite{mmsw-s-invariant}, let $\TT_n$ denote a right-handed full twist braid on $n$ strands (denoted $\FF_n$ in \cite{willis-kh}). Multiple full twists are denoted $\TT_n^k$, and the formal infinite full twist is denoted $\TT_n^\infty$.
    \item Let $C^*(\TT_n^\infty)$ denote the infinite twist complex as in \cite[Theorem 2.2]{willis-kh}.
\end{itemize}

\begin{figure}[h]
    \centering
    \includegraphics[width=\textwidth]{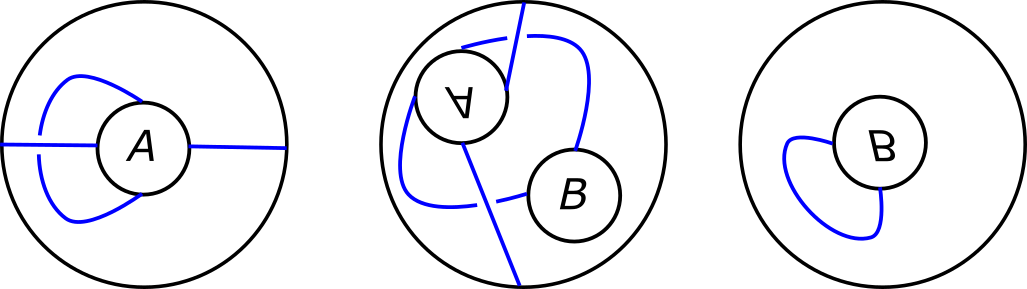}
    \caption{An example of a diagram of a link in $M=\#^3 \R P^3$}
    \label{fig:link-example}
\end{figure}

\section{Connected Sums with \texorpdfstring{$\#^q(S^1\times S^2)$}{\#^q(S^1xS^2)}}

Now we consider links in $M^q=M\#(\#^q (S^1\times S^2))$, where as before $M=M_1\#\dots\# M_r$ is a connected sum of interval bundles over surfaces. 

Recall from \cite{willis-kh} that a link in $\#^q (S^1\times S^2)$ can be represented as a diagram drawn in $P'$, a fixed planar projection of $S^1\times S^2$, which is the plane with $q$ pairs of surgery spheres connected by surgery lines. Let $\overrightarrow{P'}$ be the half plane $[0,\infty)\times\R$ together with $q$ pairs of surgery spheres (not to be confused with separating spheres) in $[1,\infty)\times\R$ connected by surgery lines. We represent links in $M^q$ as link diagrams drawn in $\Sigma^q$, where
$$\Sigma^q=\left((\Sigma\setminus D_r)\sqcup \overrightarrow{P'}\right)/\partial D_r\sim\partial\overrightarrow{P'}.$$ 

\begin{figure}[h]
    \centering
    \includegraphics[width=\textwidth]{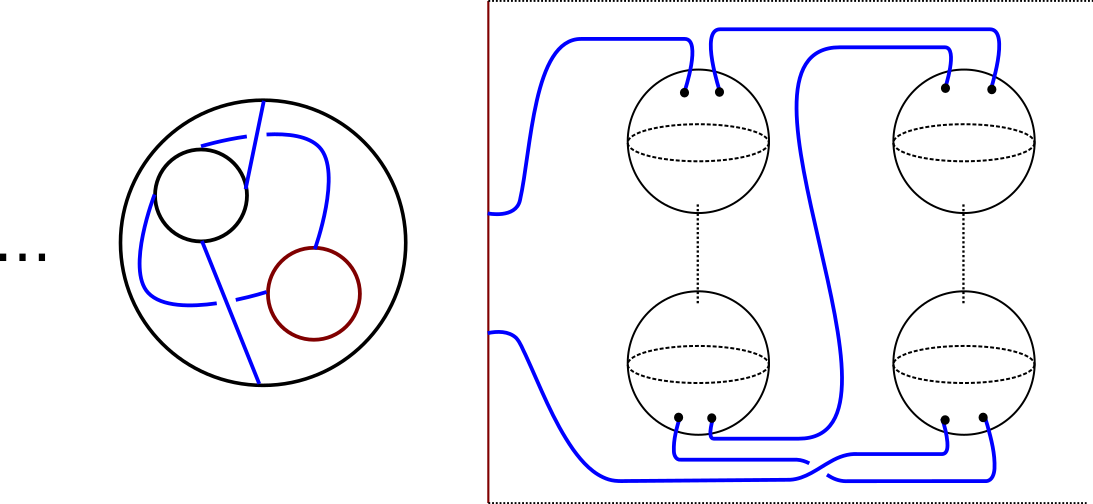}
    \caption{Link diagrams of links in $M^q$ are drawn in $\Sigma^q$}
    \label{fig:s1s2-connect-sum-example}
\end{figure}

As before, we isotope the part of the link in $M\setminus B^3$ to be disjoint from $D_i\times I$ except the intersections with the link and the separating spheres at $\partial D_i\times\{1/2\}$. We isotope the part of the link in $(\#^q (S^1\times S^2))\setminus B^3$ to be in standard position as described in \cite{willis-kh}.

\begin{proposition}\label{prop:moves-sigmaq}
Two link diagrams $L$ and $L'$ in $\Sigma^q$ represent isotopic links if and only if they are related by the following moves:

\begin{itemize}
    \item Reidemeister moves away from separating spheres, surgery spheres, and surgery lines,
    \item Finger, mirror, and handleslide moves across $\partial D_i$ for $i=1,\dots,r$,
    \item Finger and mirror moves across $\partial\overrightarrow{P'}$,
    \item Reidemeister moves in $\overrightarrow{P'}$ involving surgery spheres and surgery lines,
    \item Point pass moves and surgery wrap moves in $\overrightarrow{P'}$.
\end{itemize}
\end{proposition}

Given a link diagram $L$ in $\Sigma^q$ and a vector $\vec k=(k_1,\dots,k_q)\in\Z^q,$ we define the link diagram $L(\vec k)$ in $\Sigma\#\overrightarrow{P}$ to be the link diagram obtained by performing the replacements as described in Definition 3.5 of \cite{willis-kh}, that is, for the $i$-th pair of surgery spheres and surgery line, we remove them and connect the $n_i$ ends of the tangles with arcs parallel to the surgery line, and we insert $k_i$ full twists on $n_i$ strands at some point on the surgery line.

Say that a link $\LL\subset M^q$ is even if it has even geometric intersection numbers with each surgery sphere, that is, $n_i$ is even for all $i=1,\dots,q$. From now on, we assume all links in $M^q$ are even.

Given an even link $L$, we construct the Khovanov chain complex of $L$, denoted $CKh(L)$, as follows:

\begin{enumerate}
    \item Replace $L$ by $L(\vec 0)$, viewed as a link in $M\#S^3$.
    \item For each $i=1,\dots, q$, choose a point along $sl_i$ that is not a crossing point, and insert $C^*(\TT_{n_i}^\infty)$.
    \item Take the planar algebraic tensor product of all of these crossings in the sense of Bar Natan's tangle canopolies \cite{bar-natan}.
    \item Apply the Khovanov functor.
\end{enumerate}

Steps 1-3 of the above are analogous to Definition 3.6 of \cite{willis-kh}. Applying only steps 1-3 gives an object $BN(L)$ in Bar Natan's category of chain complexes of crossingless diagrams. However, we do not claim $BN(L)$ is an invariant of the link.

The following proposition and its proof are analogous to \cite[Proposition 3.7]{willis-kh}, with the extra step of applying the Khovanov functor at the end. 

\begin{proposition}\cite[Proposition 3.7]{willis-kh}\label{prop:truncation-approx}
Given a link diagram $L\subset\Sigma^q$ and an arbitrary homological lower bound $a$, there exists some finite $\vec k$ such that the truncated Khovanov complex $CKh_{\ge a}(L)=CKh_{\ge a}(L(\vec\infty))$ is equal to the truncation of the finite approximation complex $CKh_{\ge a}(L(\vec k))$. In other words, $CKh(L)$ can be approximated by the truncation of a finite complex in any given homological range.
\end{proposition}

\begin{theorem}
The Khovanov chain complex $CKh(L)$ is an invariant of the link $\LL\subset M^q$ up to an overall grading shift, that is, given two diagrams $L_1$ and $L_2$ of $\LL$, there is a chain homotopy equivalence $CKh(L_1)\simeq CKh(L_2)$ up to an overall grading shift.
\end{theorem}

\begin{proof}
The theorem follows once we show invariance under each of the moves in Proposition~\ref{prop:moves-sigmaq}. By Proposition~\ref{prop:truncation-approx}, it suffices to show invariance of the finite approximation complexes, that is, if $L_1$ and $L_2$ represent the same link in $M^q$, then there is a chain homotopy equivalence $CKh(L_1(\vec k))\simeq CKh(L_2(\vec k))$ for any vector $\vec k$.

A diagram move supported in $\Sigma\setminus D_r\subset \Sigma\#\overrightarrow{P'}$, including a handleslide over $\partial D_r$, induces the same diagram move from $L_1(\vec k)$ to $L_2(\vec k)$, which we view as link diagrams of links in $M\# S^3$. By \cite[Theorem 4.6]{du-kh-connect-sums}, we get a chain homotopy equivalence $CKh(L_1(\vec k))\simeq CKh(L_2(\vec k))$.

A diagram move supported in $\overrightarrow{P'}\subset \Sigma\#\overrightarrow{P'}$ induces a diagram move supported in $\overrightarrow{P}\subset\Sigma\#\overrightarrow{P}$ from $L_1(\vec k)$ to $L_2(\vec k)$. For $i=1,2$, let 
$$\overleftarrow{L}_i(\vec k)=L_i(\vec k)\cap(\Sigma\setminus D_r),\hspace{5mm}\overrightarrow{L}_i(\vec k)=L_i(\vec k)\cap\overrightarrow{P}.$$ By assumption, $\overleftarrow{L}_1(\vec k)=\overleftarrow{L}_2(\vec k)$, while $\overrightarrow{L}_1(\vec k)$ differs from $\overrightarrow{L}_2(\vec k)$ by a diagram move.

The proof of \cite[Theorem 3.9]{willis-kh} is applied locally near the diagram moves, meaning it can be extended to show $BN(\overrightarrow{L}_1(\vec k))$ is chain homotopy equivalent to $BN(\overrightarrow{L}_2(\vec k))$ up to an overall grading shift. Thus, since $BN(L_i(\vec k))=BN(\overleftarrow{L}_i(\vec k))\otimes BN(\overrightarrow{L}_i(\vec k))$, we get that $BN(L_1(\vec k))$ is chain homotopy equivalent to $BN(L_2(\vec k))$ up to an overall grading shift, and hence also the Khovanov homologies.

Lastly, finger and mirror moves across $\partial D_r=\partial\overrightarrow{P'}$ are dealt with the same way as in the proof of \cite[Theorem 3.2]{du-kh-connect-sums}.
\end{proof}

\section{Lee Homology and s-invariant}

Lee homology is a deformation of the usual Khovanov functor. It results in a chain complex that is homologically graded and quantum filtered.

Let $R$ be a commutative unital ring such that $2$ is invertible. As in \cite{bar-natan}, we will quotient the Bar-Natan category $Cob^3_{/l}$ by the additional relation that a genus $3$ surface is equal to $8$.

\begin{definition}
Let $V=\langle v_+,v_-\rangle$ be the two-dimensional graded $R$-module as before. Let $\FF'$ be the TQFT defined by $\FF'(O)=V$, $\FF'(f)=\FF(f)$ for $f$ a birth, death, or 1-1 bifurcation. If $S$ is a $1$-handle cobordism that splits circle $C$ into two circles $C_1$ and $C_2$, then 
$$\FF'(S)=\Delta':\begin{cases}
v_+\mapsto v_+\otimes w_-(S,C_2)+w_-(S,C_1)\otimes v_+ \\
w_-(S,C)\mapsto w_-(S,C_1)\otimes w_-(S,C_2)+v_+\otimes v_+.
\end{cases}$$ If $S$ is instead a $1$-handle cobordism that merges circles $C_1$ and $C_2$ into $C$, then
$$\FF(S)=m':\begin{cases}
        v_+\otimes v_+\mapsto v_+,\\ v_+\otimes w_-(S,C_2)\mapsto w_-(S,C) \\
        w_-(S,C_1)\otimes v_+\mapsto w_-(S,C),\\ w_-(S,C_1)\otimes w_-(S,C_2)\mapsto v_+.
    \end{cases}$$
\end{definition}

As in Section 4.2 of \cite{du-kh-connect-sums}, $w_-(S,C)$ is equal to $(-1)^{o(S,C)} v_-$, where the sign is determined by some additional orientation choices.

For $\LL$ a link in $M^q$ with link diagram $L$, let $CLee(L)$ be the chain complex $CLee(L)=\FF'(BN(L))$. 

\begin{proposition}
$CLee(L)$ is a chain complex.
\end{proposition}
\begin{proof}
The proof involves a similar checking of cases to $CKh(L)$ so we omit it.
\end{proof}

From looking at the maps in the differential, it is clear that $CLee(L)$ is homologically graded and quantum filtered, with the differential increasing quantum filtration.

The following proposition and its proof are similar to \cite[Corollary 2.2]{mmsw-s-invariant}.

\begin{proposition}
For any fixed homological degree $d$, there exists $\vec k\in\Z^q$ such that $CLee(L)_{\ge d}=CLee(L(\vec k))$.
\end{proposition}

\begin{proposition}
The homotopy type of $CLee(L)$ is an invariant of the link.
\end{proposition}
\begin{proof}
The proof is mostly the same as in the case of $CKh(L)$ so we omit it.
\end{proof}

From now on, consider Lee homology with coefficients in a field $\F$ where $2$ is invertible. We do a change of basis of $V$ and write $a=v_++v_-$ and $b=v_+-v_-$. Then $V$ is freely generated as an $\F$-vector space by $a$ and $b$. The effect of the Gabrov\v sek signs is to swap the roles of $a$ and $b$. More precisely, for $S$ a $1$-handle cobordism and $C$ a boundary component of $S$, let $sw,sw_{(S,C)}:V\to V$ be the linear maps given by $$sw(x)=\begin{cases}
    b & x=a \\
    a & x=b
\end{cases}$$ and $$sw_{(S,C)}(x)=\begin{cases}
    x & \text{if $C$ is locally consistently oriented with $S$} \\
    sw(x) & \text{otherwise.}
\end{cases}$$ For $C_1,\dots,C_n$ (distinct) boundary components of $S$, let $sw_{(S,C_1,\dots,C_n)}:V^{\otimes n}\to V^{\otimes n}$ be defined by
$$sw_{(S,C_1,\dots,C_n)}(x_1\otimes\cdots\otimes x_n)=sw_{(S,C_1)}(x_1)\otimes\cdots\otimes sw_{(S,C_n)}(x_n).$$ The Lee split and merge maps are given by 
\begin{align*}
    m'(x\otimes y) &= sw_{(S,C)}(m''(sw_{(S,C_1)}(x)\otimes sw_{(S,C_2)}(y)) \\
    \Delta'(x) &= sw_{(S,C_1,C_2)}(\Delta'(sw_{(S,C)}(x))),
\end{align*} where 
$$m'':\begin{cases}
    a \otimes a \mapsto 2a \\
    a \otimes b \mapsto 0 \\
    b \otimes a \mapsto 0 \\
    b \otimes b \mapsto 2b
\end{cases}\hspace{5mm} \Delta'':\begin{cases}
    a\mapsto a\otimes a \\
    b\mapsto -b\otimes b.
\end{cases}$$

For now, consider links $\LL\subset M$. 

\begin{definition}
A connected component of the projection of a link diagram $L$ to $\Sigma$ is called a projection component.
\end{definition}

\begin{definition}\label{def-basepoints}
Given a link diagram $L$, a set of oriented basepoints $P=\{\vec p_1,\dots,\vec p_l\}$ for $L$ consists of:
\begin{enumerate}
    \item A point $p_i$ on each projection component of $L$, and
    \item A choice of local orientation for the link at each point $p_i$.
\end{enumerate}
\end{definition}

Let $P=\{\vec p_1,\dots,\vec p_l\}$ be a set of oriented basepoints for $L$ as in Definition~\ref{def-basepoints}. For an oriented link diagram $(L,o)$, take the oriented resolution $L_o$, considered as a diagram of oriented circles. Consider the circles as vertices, and for every crossing of $L$, put an edge between the arcs of $L_o$ that it adjoins. Then the connected components of the resulting graph are the projection components of $L$.

In the above graph, call an edge \textit{active} if changing the resolution of the crossing does not correspond to a 1-1 bifurcation. Since we took the oriented resolution $L_o$, each active edge connects two distinct circle components of $L_o$. Additionally, since an edge that is not active connects one circle component to itself, a non-active edge is always between the same projection component of $L$.

\begin{proposition}\label{prop:lee-exist}
Given an oriented link $(L,o)$ and $P$ a set of oriented basepoints for $L$, there exists a unique way to label the (oriented) circles of $L_o$ with either $a$ or $b$ satisfying the following:
\begin{enumerate}
    \item The circles in $L_o$ containing a basepoint $p_i$ are labeled $a$ if its orientation agrees with the local orientation at $p_i$, and $b$ otherwise.
    \item No two circles that share an active edge have the same label.
\end{enumerate}
\end{proposition}
\begin{proof}
A labeling satisfying condition (2) exists if and only if there are no odd-length cycles in the graph of active edges for $L_o$. Suppose we have a length $2k+1$ cycle. Going around the cycle, starting from any one of the circles and in either direction, number the circles $1$ through $2k+1$. Since $L_o$ is the oriented resolution, choosing an orientation on circle $1$ determines the orientation on circle $2$, and so on, until going around the cycle shows that circle $2k+1$ determines an orientation on circle $1$ that is opposite to the orientation we started with. Hence an odd length cycle cannot exist.

\begin{figure}[h]
    \centering
    \def\svgwidth{.9\linewidth}
    %% Creator: Inkscape 1.1 (c68e22c387, 2021-05-23), www.inkscape.org
%% Accompanies image file 'no-odd-cycles.pdf' (pdf, eps, ps)
\begingroup%
  \makeatletter%
  \providecommand\color[2][]{%
    \errmessage{(Inkscape) Color is used for the text in Inkscape, but the package 'color.sty' is not loaded}%
    \renewcommand\color[2][]{}%
  }%
  \providecommand\transparent[1]{%
    \errmessage{(Inkscape) Transparency is used (non-zero) for the text in Inkscape, but the package 'transparent.sty' is not loaded}%
    \renewcommand\transparent[1]{}%
  }%
  \providecommand\rotatebox[2]{#2}%
  \newcommand*\fsize{\dimexpr\f@size pt\relax}%
  \newcommand*\lineheight[1]{\fontsize{\fsize}{#1\fsize}\selectfont}%
  \ifx\svgwidth\undefined%
    \setlength{\unitlength}{654.75293257bp}%
    \ifx\svgscale\undefined%
      \relax%
    \else%
      \setlength{\unitlength}{\unitlength * \real{\svgscale}}%
    \fi%
  \else%
    \setlength{\unitlength}{\svgwidth}%
  \fi%
  \global\let\svgwidth\undefined%
  \global\let\svgscale\undefined%
  \makeatother%
  \begin{picture}(1,0.28597065)%
    \lineheight{1}%
    \setlength\tabcolsep{0pt}%
    \put(0,0){\includegraphics[width=\unitlength,page=1]{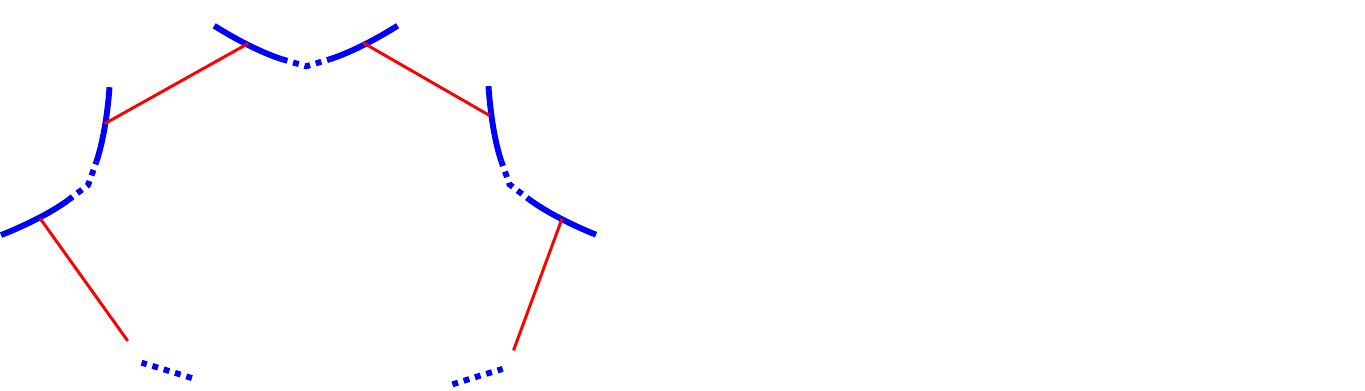}}%
    \put(0.19951585,0.26246463){\makebox(0,0)[lt]{\lineheight{1.25}\smash{\begin{tabular}[t]{l}$1$\end{tabular}}}}%
    \put(0.38574232,0.17128317){\makebox(0,0)[lt]{\lineheight{1.25}\smash{\begin{tabular}[t]{l}$2$\end{tabular}}}}%
    \put(0.01406203,0.15814685){\makebox(0,0)[lt]{\lineheight{1.25}\smash{\begin{tabular}[t]{l}$2k+1$\end{tabular}}}}%
    \put(0,0){\includegraphics[width=\unitlength,page=2]{no-odd-cycles.pdf}}%
    \put(0.94386001,0.16915176){\makebox(0,0)[lt]{\lineheight{1.25}\smash{\begin{tabular}[t]{l}$2$\end{tabular}}}}%
    \put(0.57217978,0.15601548){\makebox(0,0)[lt]{\lineheight{1.25}\smash{\begin{tabular}[t]{l}$2k+1$\end{tabular}}}}%
    \put(0,0){\includegraphics[width=\unitlength,page=3]{no-odd-cycles.pdf}}%
    \put(0.75763358,0.26033318){\makebox(0,0)[lt]{\lineheight{1.25}\smash{\begin{tabular}[t]{l}$1$\end{tabular}}}}%
    \put(0,0){\includegraphics[width=\unitlength,page=4]{no-odd-cycles.pdf}}%
  \end{picture}%
\endgroup%

    \caption{An odd cycle in the graph (left) and the portion of the link that would correspond to it (right). Starting with an orientation on circle $1$ and going around, one arrives at an orientation on circle $1$ different from the original one.}
    \label{fig:no-odd-cycles}
\end{figure}

Any labeling satisfying (2) can be made to also satisfy (1) by flipping the overall $a$ or $b$ parity within projection components. The definition of projection components makes such a labeling unique.
\end{proof}

\begin{definition}\label{def-lee-generator}
With $L,o,P$ as above, the Lee generator $\frak s(L,o,P)=\frak s(L_o,P)$ is the element of $CLee(L)$ given by taking the oriented resolution $L_o$ and assigning $a$ and $b$ labels to the circles according to the unique way given by Proposition~\ref{prop:lee-exist}.
\end{definition}

Let $o$ and $o'$ be two distinct orientations for $L$, then the oriented resolutions $L_o$ and $L_{o'}$ are distinct: either they are distinct as diagrams of unoriented circles, or $L_o$ and $L_{o'}$ have the same set of circles but there exists a circle such that $L_o$ and $L_{o'}$ differ on its orientation. In the later case, $L_o$ and $L_{o'}$ must differ by their orientations on a basepointed circle. Therefore, $\frak s(L,o,P)$ and $\frak s(L,o',P)$ are distinct. 

These elements are called the Lee generators for $(L,o)$ since they are generators of the Lee homology, as explained in the following lemma.

\begin{lemma}\label{lem:lee-generator}
The rank of $HLee(L)$ for $\LL$ an $n$-component link is at least $2^n$.
\end{lemma}
\begin{proof}

Fix a diagram $L$ and a set of basepoints with local orientations $P$. The $2^n$ orientations $o$ of $L$ give $2^n$ distinct elements $\frak s(L,o,P)$ in $CLee(L)$. These chain elements are in the kernel of the Lee differential and not in the image: for crossings that do not correspond to 1-1 bifurcations, the argument follows as in \cite{lee-hlee}, and for 1-1 bifurcations, the differential is zero. Hence, they generate $2^n$ distinct nonzero homology classes in $HLee(L)$.
\end{proof}

\begin{theorem}\cite[Theorem 4.2]{lee-hlee}
The dimension of $HLee(\LL)$ for a link $\LL$ of $n$ components equals $2^n$.
\end{theorem}

The proof is the same as in \cite{lee-hlee}.

It is easily observed that the homological degrees of all Lee generators are $0$, thus $HLee(\LL)$ is supported in homological degree $0$.

%Since $HLee(L)$ is an invariant of the link, so are the following two quantities:
%$$s_{\min}(L)=\min\{q(x):x\in HLee(L), x\ne 0\},$$
%$$s_{\max}(L)=\max\{q(x):x\in HLee(L), x\ne 0\}.$$

%\cite{rasmussen-s-inv} showed that for knots $K$ in $S^3$, $$s_{\max}(K)=s_{\min}(K)+2.$$ The same argument also holds for knots in $M$, and we also define $$s(K)=s_{\max}(K)-1=s_{\min}(K)+1.$$

\subsection{Cobordism Maps}

Let $\Sigma$ be an oriented cobordism in $S^3\times I$ between oriented links $L_0$ and $L_1$ in $S^3$, then $\Sigma$ induces a map $\phi_\Sigma:HLee(L_0)\to HLee(L_1)$ that is filtered of degree $\chi(\Sigma)$. Furthermore, if $\Sigma$ is weakly-connected in the sense that every component of $\Sigma$ has a boundary component in $L_0$, then $\phi_\Sigma$ preserves Lee generators, that is, $\phi_\Sigma([\frak s_{o_0}(L_0)])$ is a nonzero multiple of $[\frak s_{o_1}(L_1)]$ \cite[Proposition 4.1]{rasmussen-s-inv}\cite[Equation (7)]{bw-categorification}. 

The same is true for cobordisms in $M\times I$ between links in $M$, that is, an oriented cobordism $\Sigma$ in $M\times I$ between oriented links $L_0$ and $L_1$ in $M$ induces a map $\phi_\Sigma:HLee(L_0)\to HLee(L_1)$ that is filtered of degree $\chi(\Sigma)$, and if $\Sigma$ is weakly-connected, then $\phi_\Sigma$ maps $\frak s(L_0,o_0,P_0)$ to a nonzero multiple of $\frak s(L_1,o_1,P_1)$ for suitable choices of $P_0$ and $P_1$. 

We break cobordisms into elementary parts: morse moves and diagram moves. For $0$ and $2$-handle attachments, the rest of the link remains the same, and the circle that is being created/capped off has a basepoint on it with either local orientation. For an oriented $1$-handle attachment, if the saddle does not change the number of projection components, then the basepoint is chosen away from the attaching region. Otherwise, suppose the saddle joins two disjoint projection components into one, then put the two basepoints in $L_0$ on the feet of the saddle with local orientations induced by the orientation of the saddle, and put one basepoint in $L_1$ anywhere on the saddle with local orientation induced by the orientation of the saddle as in Figure~\ref{fig:saddle-loc-orientation}. The placement of the basepoints is similar if the saddle separates one link component into two disjoint links. The construction of cobordism maps for the $0,1,$ and $2$-handle attachment moves and Reidemeister moves is the same as for the $S^3$ case and satisfy the desired properties.

\begin{figure}
    \centering
    \includegraphics[width=0.6\textwidth]{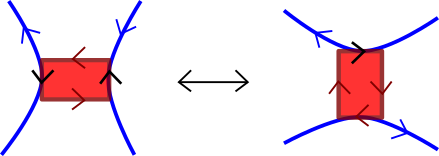}
    \caption{Basepoints and local orientations (shown in black) on a $1$-handle attachment (shown in red) that changes projection components of a link (shown in blue).}
    \label{fig:saddle-loc-orientation}
\end{figure}

For finger and mirror moves, put the basepoint away from the region where the move is happening. Finger and mirror moves induce isomorphisms preserving Lee generators.

For Reidemeister moves, the cobordism maps are the same as in the $S^3$ case. We detail how to choose the basepoints as follows:

For a Reidemeister I move, put the basepoint away from the region where the move is happening. 

For a Reidemeister II move, if the move does not change the number of projection components, put the basepoint away from the region where the move is happening. Otherwise, suppose $L_1$ is obtained from $L_0$ by performing a Reidemeister II move that moves an arc of the link over another arc in a different projection component, creating two new crossings. Then two basepoints are placed on the two arcs where the move occurs in $L_0$ and one basepoint is placed in the region where the move occurs in $L_1$, with local orientations as in Figure~\ref{fig:r2-loc-orientation}.

\begin{figure}[h]
    \centering
    \includegraphics[width=0.6\textwidth]{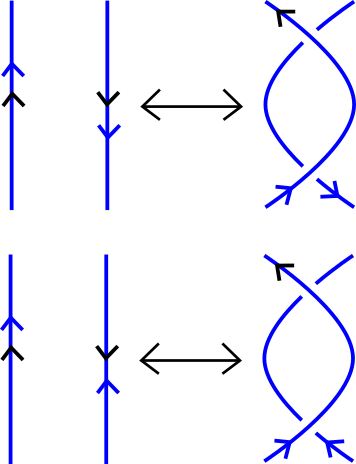}
    \caption{Basepoints and local orientations for a Reidemeister II move that changes projection components. In the top row, the two strands of the link are oriented oppositely, and in the bottom row, they are oriented the same way.}
    \label{fig:r2-loc-orientation}
\end{figure}

A Reidemeister III move occurs only within a projection component, so put the basepoint away from the region where the move is happening. 

For a handleslide move between links $L_0$ and $L_1$, choose a sequence of moves as in the proof of handleslide invariance of Section 4.3 of \cite{du-kh-connect-sums} that induce homotopy equivalences on Lee homology. The cobordism map is the composition of these homotopy equivalences. The placement of basepoints for each of these moves is similar to the placement of basepoints in the moves described above, except for the move that changes the homotopy class of a strand, where the basepoint is placed away from where the move is happening. 

Note that the map depends on the choice of sequence of moves, and we do not know if it is well-defined up to homotopy. However, it is enough for our applications to know that it exists.

\subsection{s-invariant}

By the argument of \cite{bw-categorification}, for any oriented link diagram $(L,o)$, the quantum degrees of $\frak s(L,o,P)+\frak s(L,\overline o,P)$ and $\frak s(L,o,P)-\frak s(L,\overline o,P)$ differ by exactly two. Analogously to \cite{bw-categorification}, we define the Rasmussen invariant $s(L)$ of the oriented link diagram $L$ to be

\begin{equation}\label{eq:s-inv}
    s(L,o,P)=\frac{q(\frak s(L,o,P)+\frak s(L,\overline o,P))+q(\frak s(L,o,P)-\frak s(L,\overline o,P))}{2}.
\end{equation}

\begin{proposition}\cite[Corollary 3.6]{rasmussen-s-inv}\label{prop:min-s-grading}
$$q(\frak s(L,o,P))=q(\frak s(L,\overline o,P))=s(L,o,P)-1.$$
\end{proposition}

To show that the Rasmussen invariant $s(L,o,P)$ is an invariant of the oriented link $(\LL,o)$, we need to show $s(L,o,P)$ does not depend on the choice of basepoints with local orientations $P$, and a diagram move from two link diagrams $L$ to $L'$ for the same link induces a homotopy equivalence that maps the Lee generator $\frak s(L,o,P)$ of $L$ to a unit multiple of the Lee generator for $L'$. 

First, suppose $L$ has one projection component, then there is one basepoint, $P=\{\vec p\}$. Let $\overline{\vec p}$ be $\vec p$ with the local orientation reversed, and $\overline P=\{\overline{\vec p}\}.$ It is clear that $\frak s(L,o,\overline P)=\frak s(L,\overline o,P)$, and it follows that $s(L,o,P)=s(L,o,\overline P)$.

As in \cite[Lemma 6.1]{bw-categorification}, if $L=L_1\sqcup L_2$ is a disjoint union, then $$s(L_1\sqcup L_2, o_1\sqcup o_2, P_1\sqcup P_2)=s(L_1,o_1,P_1)+s(L_2,o_2,P_2)-1.$$ Now for any link diagram $L$, suppose we reverse the local orientation at point $\vec p_i\in P$ to get $P'$. Let $L=L_1\sqcup L_2$ where $L_1$ is the projection component containing $p_i$, then $P=P_1\sqcup P_2$ where $P_1=\{\vec p_i\}$ and $P_2=P\setminus\{\vec p_i\}$. By the above, $s(L_1,o_1,P_1)=s(L_1,o_1,\overline P_1)$, so
\begin{align*}
    s(L,o,P') &= s(L_1,o_1,\overline P_1)+s(L_2,o_2,P_2)-1 \\
    &= s(L_1,o_1,P_1)+s(L_2,o_2,P_2)-1 \\
    &= s(L,o,P).
\end{align*}

We showed above that diagram moves induce cobordism maps that are homotopy equivalences mapping Lee generators to unit multiples of Lee generators. Thus, the $s$-invariant is invariant under diagram moves.

%Finally, moving a basepoint within a component of the projection has the same effect as reversing the local orientation at a basepoint (or doing nothing). Therefore, $s(L,o,P)$ is invariant under the choice of $P$, and we will henceforth write $s(L,o)$ or just $s(L)$ when the orientation is clear from context.

%For Reidemeister moves, the proof of invariance is the same as in \cite{rasmussen-s-inv}. The isomorphisms induced by finger and mirror moves clearly preserve Lee generators.

%For a handleslide move between links $L$ and $L'$, choose a sequence of moves as in the proof of handleslide invariance that induce homotopy equivalences on Lee homology. This sequence of moves consists of Reidemeister and Reidemeister prime moves, finger and mirror moves, or a move that changes the homotopy class of a strand. As mentioned previously, Reidemeister moves and finger and mirror moves preserve Lee generators. 

%Reidemeister prime moves induce maps that are analogous to their corresponding Reidemeister move and preserve Lee generators. Finally, a move that changes the homotopy class of a strand maps $\frak s(L,o,P)$ to $\frak s(L',o',P')$.

\subsection{s-invariant in Connected Sums with \texorpdfstring{$\#^q(S^1\times S^2)$}{\#^q(S^1xS^2)}}

Now let $M^q=M\#(\#^q (S^1\times S^2))$ and $\LL\subset M^q$ be an even link. Recall that $CLee(L)$ is built by inserting infinite twist complexes $C^*(\TT_{n_i}^\infty)$ into the diagram. The following is shown in the proof of \cite[Theorem 2.10]{mmsw-s-invariant}:

\begin{lemma}
A null-homologous orientation $o$ on $\TT_{n_i}$ determines a specific oriented diagram $\epsilon_{n_i,o}\in CLee^\#(\TT_{n_i})$ satisfying the following properties:
\begin{enumerate}
    \item $\epsilon_{n_i,o}$ has homological degree $0$
    \item it has through-degree $0$
    \item it is in the stable range of the sequence of inclusions $$CLee^\#(\TT_{n_i})_{\ge -1}\hookrightarrow CLee^\#(\TT_{n_i}^2)_{\ge -3}\hookrightarrow CLee^\#(\TT_{n_i}^3)_{\ge -5}\hookrightarrow\cdots$$ which limit to $CLee(\TT_{n_i}^\infty)$
    \item in the simplification of the Lee complex of any finite number of full twists $CLee(\TT_{n_i}^k)\to CLee^\#(\TT_{n_i}^k)$, any resolution with through-degree $0$ is simplified to $\epsilon_{n_i,o}$.
\end{enumerate}
\end{lemma}

This diagram $\epsilon_{n_i,o}$ is the simplification of the usual oriented resolution of $\TT_{n_i}$, $\TT_{n_i,o}$. As a consequence of this lemma, for any link diagram $L$ for a link in $M^q$ with an orientation $o$, there is a diagram $L_o$ given by replacing each twist region $\TT_{n_i}^\infty$ with $\epsilon_{n_i,o}$ such that the simplification of the oriented resolution $L(\vec k)_o$ for any $k\ge 1$ results in $L_o$.

\begin{definition}
Given a link diagram $L$ of a link in $M^q$, an orientation $o$ of $L$ making $L$ have zero algebraic intersection number with each surgery sphere, and a set of basepoints $P$ for $L$ as in Definition~\ref{def-basepoints}, the Lee generator $\frak s(L,o,P)\in CLee(L)$ is the chain element constructed by the following steps:
\begin{enumerate}
    \item Construct the diagram $L_o$ as described above, considered as the oriented resolution of a link diagram for a link in $M$
    \item Take the Lee generator $\frak s(L_o,P)$.
\end{enumerate}
\end{definition}

The above definition makes sense based on the argument in the proof of \cite[Theorem 2.10]{mmsw-s-invariant}: for any $k\ge 1$, the Lee generator $\frak s^\#(L(\vec k)_{o_k},P_k)$ for the simplification of $L(\vec k)_{o_k}$ is equal to $\frak s(L,o,P)$. For $k\ge \frac{n^++2}{2}$, $\frak s^\#(L(\vec k)_{o_k},P_k)$ is in the stable range of the sequence of inclusions
\begin{equation}\label{eq:inclusion-lee}
        CLee^\#(L(\vec 1))_{\ge n^+-1}\hookrightarrow CLee^\#(L(\vec 2))_{\ge n^+-3}\hookrightarrow\cdots\hookrightarrow CLee^\#(L(\vec k))_{\ge n^+-2k+1}\hookrightarrow\cdots,
\end{equation}
so $\frak s(L,o,P)$ is well-defined.

Furthermore, they show that the Lee generator $\frak s(L(\vec k),o_k,P_k)$ for $L(\vec k)_{o_k}$ is mapped to a unit multiple of $\frak s^\#(L(\vec k)_{o_k},P_k)$ under the simplification homotopy equivalence $\Psi:CLee(L(\vec k))\to CLee^\#(L(\vec k))$.

\begin{proposition}
The Lee generator $[\frak s(L,o,P)]$ generates an $\F$-summand in $HLee(L)$.
\end{proposition}
\begin{proof}
We know that $[\frak s(L(\vec k),o_k,P_k)]$ generates an $\F$-summand in $HLee(L(\vec k))$ for any $k$ by Lemma~\ref{lem:lee-generator}. Since the simplification homotopy equivalence $$\Psi:CLee(L(\vec k))\to CLee^\#(L(\vec k))$$ maps $\frak s(L(\vec k),o_k,P_k)$ to a unit multiple of $\frak s^\#(L(\vec k)_{o_k},P_k)$ and $$\frak s^\#(L(\vec k)_{o_k},P_k)=\frak s(L,o,P),$$ the result follows.
\end{proof}

As before, define the $s$-invariant $s(L,o,P)$ of $(L,o)$ with set of oriented basepoints $P$ as in Equation~\ref{eq:s-inv}. The invariance of $s(L,o,P)$ for links in $M^q$ follows from the invariance of the $s$-invariant for links in $M$ together with the fact that finite approximations of $HLee(L)$ preserve Lee generators.

\subsection{Properties of the s-invariant}

Let $U$ be a knot with no crossings, and $[U]\in H_1(M)$ be its homology class. Then we say that $U$ is a class $[U]$-unknot.

The $s$-invariant of any unknot is zero.

\begin{proposition}
If an oriented link $L$ is a disjoint union of two oriented links $L_1$ and $L_2$, then
$$s(L)=s(L_1)+s(L_2)-1.$$
\end{proposition}

It follows that for an unlink of $n$ components, meaning a disjoint union of $n$ unknots, the $s$-invariant is $1-n$.  

Reversing the orientation of the ambient manifold $M$ to get $\overline M$ turns a link $L\subset M$ into its mirror, which we denote as $\overline L\subset\overline M$. In the diagram, the mirror of a link has all over/under-crossings reversed from the original link.

\begin{proposition}\cite[Proposition 3.9]{rasmussen-s-inv}\label{prop-mirror-s}
Let $K$ be a knot, and $\overline K$ its mirror image. Then
$$s(\overline K)=-s(K).$$
\end{proposition}

The $s$-invariant is compatible with the previous versions of the $s$-invariant defined by \cite{rasmussen-s-inv}\cite{bw-categorification} for $S^3$, \cite{mmsw-s-invariant} for $\#^q (S^1\times S^2)$, and \cite{mw-s-rp3} for $\R P^3$ in the following sense. Suppose $\LL\subset M^q$ is a local $S^3$, $\R P^3$, or $\#^q (S^1\times S^2)$ link, that is, it is disjoint from all separating spheres and entirely contained in a $S^3$, $\R P^3$, or $\#^q (S^1\times S^2)$ connect summand of $M^q=(M_1\#\dots\# M_r)\#(\#^q (S^1\times S^2))$, respectively. Then the $s$-invariant $s(L)$ agrees with the $s$-invariant of $L$ considered as a link in $S^3$, $\R P^3$, or $\#^q (S^1\times S^2)$, respectively. This is easy to see in the case of a local $S^3$ link, and by extension, a local $\#^q( S^1\times S^2)$ link since its Lee homology can be approximated using local $S^3$ links. 

In the case $\LL$ is a local local $\R P^3$ link, let $L$ be a diagram for $\LL$ in $\R P^3$ and let $L'$ be the same diagram but embedded in $M^q$ in a $\R P^3$ connect summand so that it misses all separating spheres. The oriented resolutions $L_o$ and $L'_o$ are the same, and there exists a choice of set of oriented basepoints $P$ for $L'$ such that $\frak s_o(L)$ and $\frak s(L',o,P)$ have the same labels of $a$ and $b$ for all circles in $L_o$. The same is true for $\overline o$ with the same choice of set of oriented basepoints $P$, so 
$$q(\frak s_o(L)+\frak s_{\overline o}(L))=q(\frak s(L',o,P)+\frak s(L',\overline o,P)),$$
$$q(\frak s_o(L)-\frak s_{\overline o}(L))=q(\frak s(L',o,P)-\frak s(L',\overline o,P)),$$ and it follows that $s(L)=s(L')$.

\section{Slice Genus}

As in the introduction, let $D(d)$ denote the $D^2$-bundle over $S^2$ with euler number $d$. For a null-homologous properly embedded orientable connected surface $\Sigma\subset D(d)$ with boundary a knot $K\subset \partial D(d)$, the genus bound
$$2g(\Sigma)\ge -s(K)$$ was proven for $d=0,1$ \cite[Theorem 1.15, Corollary 1.9]{mmsw-s-invariant} and for $d=2$ \cite{ren-slice-genus}. Their method of proof was to construct a cobordism between $K$ and the mirror of a link $T(d;p,q)$ for some natural numbers $p,q$, where $T(d;p,q)\subset \partial D(d)$ is the oriented link consisting of $p+q$ fibers of the $S^1$-bundle $\partial D(d)\to S^2$, $p$ of which are oriented positively and $q$ of which are oriented negatively.

\begin{proof}[Proof of Theorem~\ref{thm:slice-genus-bound}]

We follow the construction of \cite{mmsw-s-invariant} and \cite{ren-slice-genus}, with some modifications for the connected sum, to reduce to the calculation of $s$-invariants. Let $S_i$ denote the core $2$-sphere of $D(d_i)$. The second homology of $X$ is $H_2(X)=H_2(X_1)\oplus\dots\oplus H_2(X_n)$. Let $\Sigma\subset X$ be a properly embedded orientable connected surface with boundary $L\subset\partial X=(\partial X_1)\#\dots\#(\partial X_n)$. Perturb $\Sigma$ so that it intersects each $S_i$ transversely at a finite set of points, and let $p_i$ and $q_i$ be the number of positive and negative intersections of $\Sigma$ with $S_i$, respectively. 

Consider tubular neighborhoods of each of the core $2$-spheres $S_i\subset X_i=D(d_i)$. Tube these neighborhoods together through the connected sum regions as shown in Figure~\ref{fig:tube-nbhd} to form $N\subset X$. Perform the tubing away from the intersections with $\Sigma$. Then $\Sigma$ intersects $N$ in a set of disks: a disk for each intersection of $\Sigma$ with $S_i$, and a set of $k$ disks for intersections of $\Sigma$ with the tubes. 

Let $X_0=X\setminus N$. Then $X_0$ is homeomorphic to $(\partial X)\times I$. Let $\Sigma_0$ be $\Sigma$ with the disks that it intersects $N$ in cut out, then $\Sigma_0$ is a properly embedded surface $\Sigma_0\subset X_0$ whose boundary on $(\partial X)\times\{1\}$ is the original boundary of $\Sigma$ and whose boundary on $(\partial X)\times \{0\}$ is 
$$L'=\overline{T}(d_1;p_1,q_1)\sqcup \dots\sqcup \overline{T}(d_n;p_n,q_n)\sqcup U^k$$ for $k$ unknots disjoint from the rest of the link. Here we abuse notation a little bit to also let $T(d_i;p_i,q_i)$ denote the link $T(d_i;p_i,q_i)$ in the $\partial D(d_i)$ connected summand of $\partial X$. Thus $\Sigma_0$ is a cobordism from $L'$ to $L$. The Euler characteristic of $\Sigma_0$ is
$$\chi(\Sigma_0)=\chi(\Sigma)-p_1-q_1-\dots-p_n-q_n-k.$$

\begin{figure}[h]
    \centering
    \def\svgwidth{.9\linewidth}
    %% Creator: Inkscape 1.1 (c68e22c387, 2021-05-23), www.inkscape.org
%% Accompanies image file 'tubeNbhd.pdf' (pdf, eps, ps)
\begingroup%
  \makeatletter%
  \providecommand\color[2][]{%
    \errmessage{(Inkscape) Color is used for the text in Inkscape, but the package 'color.sty' is not loaded}%
    \renewcommand\color[2][]{}%
  }%
  \providecommand\transparent[1]{%
    \errmessage{(Inkscape) Transparency is used (non-zero) for the text in Inkscape, but the package 'transparent.sty' is not loaded}%
    \renewcommand\transparent[1]{}%
  }%
  \providecommand\rotatebox[2]{#2}%
  \newcommand*\fsize{\dimexpr\f@size pt\relax}%
  \newcommand*\lineheight[1]{\fontsize{\fsize}{#1\fsize}\selectfont}%
  \ifx\svgwidth\undefined%
    \setlength{\unitlength}{868.91699603bp}%
    \ifx\svgscale\undefined%
      \relax%
    \else%
      \setlength{\unitlength}{\unitlength * \real{\svgscale}}%
    \fi%
  \else%
    \setlength{\unitlength}{\svgwidth}%
  \fi%
  \global\let\svgwidth\undefined%
  \global\let\svgscale\undefined%
  \makeatother%
  \begin{picture}(1,0.55771283)%
    \lineheight{1}%
    \setlength\tabcolsep{0pt}%
    \put(0,0){\includegraphics[width=\unitlength,page=1]{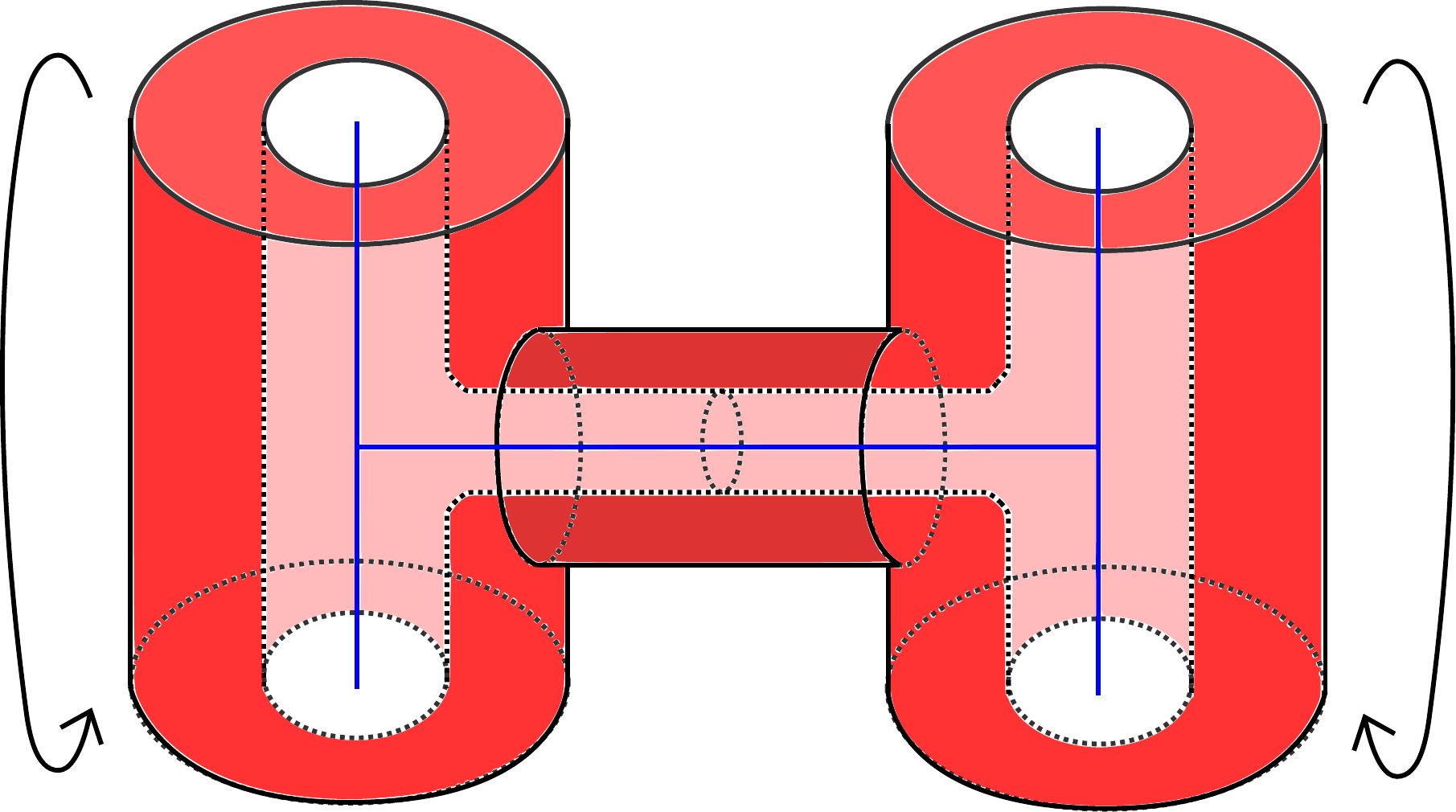}}%
    \put(0.2314802,0.05650887){\color[rgb]{0,0,1}\makebox(0,0)[lt]{\lineheight{1.25}\smash{\begin{tabular}[t]{l}$S_1$\end{tabular}}}}%
    \put(0.7337386,0.05399695){\color[rgb]{0,0,1}\makebox(0,0)[lt]{\lineheight{1.25}\smash{\begin{tabular}[t]{l}$S_2$\end{tabular}}}}%
  \end{picture}%
\endgroup%

    \caption{The neighborhood $N$ removed from $X$ to form $X_0$.}
    \label{fig:tube-nbhd}
\end{figure}

Turning the cobordism $\Sigma_0$ upside-down gives a cobordism from $\overline L$ to $\overline {L'}$. Every component of $\Sigma_0$ has a boundary in $L$, so by Theorem~\ref{thm:cobordism-s-ineq}, 
$$s(\overline{L'})-s(\overline L)\ge\chi(\Sigma_0)=\chi(\Sigma)-p_1-q_1-\dots-p_n-q_n-k.$$

Suppose $\Sigma$ is null-homologous, then $p_i=q_i$ for all $i=1,\dots,n$. It was shown by \cite[Theorem 1.6]{mmsw-s-invariant}\cite[Theorem 1.1]{ren-lee-torus}\cite[Proposition 7]{ren-slice-genus} that 
$$s(T(d;p,p))=-2p+1$$ for $d=0,1,2$. Therefore,
$$\chi(\Sigma_0)\le s(\overline{L'})-s(\overline L)=-2p_1-\dots-2p_n-(n-1+k)+n-s(\overline L).$$ If $L$ is a knot $K$, then by Proposition~\ref{prop-mirror-s}, $s(\overline K)=-s(K),$ so
$$\chi(\Sigma)\le s(K)+1.$$

\end{proof}

\bibliographystyle{alpha} %\bibliographystyle{alphahack}
\bibliography{biblio}

\end{document}